\documentclass[11pt,reqno]{amsart}

\usepackage{eucal,mathrsfs}

\usepackage[utf8]{inputenc}

\numberwithin{equation}{section}

\usepackage[a4paper,margin=1.2in]{geometry}
                		
\usepackage{graphicx}				
										
\usepackage{amssymb}

\usepackage[utf8]{inputenc}

\usepackage[polish, english]{babel}

\usepackage[T1]{fontenc}

\usepackage{mathtools}

\usepackage{tikz-cd}

\usepackage{amsthm}

\usepackage[toc]{appendix}

\usepackage{calligra}

\usepackage{hyperref}

\usepackage{tikz}

\usepackage{spectralsequences}

\newtheorem{thm}{Theorem}[section]
\newtheorem{lem}[thm]{Lemma}

\theoremstyle{remark}

\newtheorem{rmk}[thm]{Remark}
\newtheorem{ex}[thm]{Example}

\newcommand{\cE}{{\mathcal E}}
\newcommand{\cF}{{\mathcal F}}
\newcommand{\cG}{{\mathcal G}}
\newcommand{\cH}{{\mathcal H}}

\newcommand{\cL}{{\mathcal L}}

\newcommand{\cO}{{\mathcal O}}
\newcommand{\cP}{{\mathcal P}}

\newcommand{\cU}{{\mathcal U}}
\newcommand{\cV}{\mathcal{V}}

\newcommand{\PP}{{\mathbb P}}
\newcommand{\TT}{{\mathbb T}}

\newcommand{\ZZ}{{\mathbb Z}}

\newcommand{\Frob}{{\mathsf{F}_{*}}}
\newcommand{\Frobr}{{\mathsf{F}^{r}_{*}}}
\newcommand{\Fm}{{\mathsf{F}_{m}}}
\newcommand{\Frobm}{{\mathsf{F}_{m*}}}
\newcommand{\rk}{{\textnormal{rk }}}

\title{Frobenius pushforwards of vector bundles on projective spaces}

\author{Feliks Rączka}

  \address{Institute of Mathematics, Polish Academy of Sciences, ul.\ Śniadeckich 8,
    \newline\indent 00-656 Warsaw, Poland
  }
\email{fraczka@impan.pl}

\date{\today}

\begin{document}

\begin{abstract}
We investigate when the filtration induced by Beilinson's spectral sequence splits non-canonically into a direct sum decomposition. We conclude that for any vector bundle $\cE$ on a projective space over an algebraically closed field of characteristic $p>0$ there exists $r_{0}$ such that for $r\geq r_{0}$ the Frobenius pushforward $\Frobr\cE$ decomposes as a direct sum of line bundles and exterior powers of the cotangent bundle (we also give a variant for the "toric Frobenius map" valid in any characteristic).  As an application we give a short proof of Klyachko's theorem for vanishing of the cohomology of toric vector bundles on projective spaces.
\end{abstract}

\maketitle

\section{Introduction}

Let $X$ be a smooth projective variety over an algebraically closed field $k$ of characteristic $p>0$ and let $\mathsf{F}:X\to X^{(1)}$ be the $k$-linear Frobenius morphism. Let $\cE$ be a vector bundle on $X$. Then its pushforward $\Frob\cE$ is a vector bundle on $X^{(1)}$ and  it is a classical problem to study decomposition of $\Frob\cE$ into a direct sum of indecomposable vector bundles. A well known theorem of R.\ Hartshorne \cite[Corollary 6.4]{Hartshorne} states that if $\cL$ is a line bundle on $\PP^{n}_{k}$ then $\Frob\cL$ is a direct sum of line bundles. This result has been later generalized by J.\ Thomsen \cite{Thomsen} who proved that if $X$ is toric then $\Frob\cL$ is a direct sum of line bundles for any line bundle $\cL$, and P.\ Achinger \cite{Achinger1} showed that the property that the Frobenius pushforward of any line bundle splits as a direct sum of line bundles characterizes smooth projective toric varieties among connected projective $k$-schemes. Frobenius pushforwards of line bundles on quadrics have been studied by A.\ Langer \cite{Langer} and later by P.\ Achinger \cite{Achinger2}. Finally, for vector bundles of higher ranks B\o gvad \cite{Bogvad} proved that if $X$ is toric and $\cE$ is a toric bundle then for any $r\geq1$ there exists a decomposition $\Frobr\cE=\bigoplus_{\lambda}\cE_{\lambda}$ where $\rk\cE_{\lambda}=\rk\cE.$\newline

In this paper we study Frobenius pushforwards of arbitrary (i.e., of arbitrary finite rank and not necessarily toric) vector bundles on projective spaces. If $k$ is of arbitrary characteristic and $m$ is a positive integer then we have a well defined \textit{Frobenius morphism}
\[
\mathsf{F}_{m}:\PP^{n}_{k}\to\PP^{n}_{k}\; ;\;
[x_{0}:\dots:x_n]\mapsto[x_{0}^{m}:\dots:x_{n}^{m}]
\]
If $\textnormal{char }k=p$ and $m=p^{r}$ then $\Fm=\mathsf{F}^{r}$. If $m$ is coprime to the characteristic, then $\Fm$ is a finite ramified cover and if we denote by $\mu_m\subset k^{\times}$ the group of $m$-th roots of unity then $\Fm$ corresponds to the quotient of $\PP^{n}_{k}$ by the diagonal action of $\mu_{m}^{n}\subset\textnormal{PGL}_{n+1}(k)$. In any case $\Fm$ is a finite surjective morphism of smooth projective varieties and therefore it is flat. In particular if $\cE$ is a vector bundle then so is its pushforward $\Frobm\cE$ and the problem of decomposing it into a direct sum of indecomposable vector bundles is meaningful.

Recall that by Serre's vanishing theorem combined with Serre's duality, for any vector bundle $\cE$ there exists an integer $m(\cE)\geq0$ such that $\cE(m)$ has only non-zero cohomology in degree $0$ for $m\geq m(\cE)$ and has only non-zero cohomology in degree $n$ for $m\leq -m(\cE)$. The main theorem of this paper is the following.

\begin{thm}\label{MainThm}
Let $k$ be an algebraically closed field of arbitrary characteristic and let $\cE$ be a vector bundle on $\PP^{n}_{k}$. Then for $m\geq m(\cE)$ there exists a decomposition
\begin{equation}\label{Direct1}
\Frobm\cE=\bigoplus_{i=0}^{n}(\Omega^{i}_{\PP^{n}_{k}})^{\oplus a_{i}}\oplus\bigoplus_{j=1}^{n}\cO_{\PP^{n}_{k}}(-j)^{\oplus b_{j,m}}
\end{equation}
Moreover, $a_{i}=\dim_{k}H^{i}(\PP^{n}_{k},\cE).$

\end{thm}

A well known theorem of Klyachko \cite[Corollary 4.2]{Klyachko} states that if $\cE$ is a toric bundle on $\PP^{n}_{k}$ and $\rk\cE<{n \choose i}$ then $H^{i}(\PP^{n}_{k},\cE)=0$. As an application of Theorem \ref{MainThm} we give a new short proof of the theorem of Klyachko.\newline

Let us briefly explain the idea behind the proof of Theorem \ref{MainThm}. If $\cE$ is a line bundle then this follows from the theorem of Hartshorne. One classical proof of this theorem follows from Horrocks' criterion \cite[Theorem 2.3.1]{Okonek}, which states that a vector bundle $\cV$ on $\PP^{n}_{k}$ is a direct sum of line bundles if and only if $H^{i}(\PP^{n}_{k},\cV(j))=0$ for all $0<i<n$ and $j\in\ZZ.$ Since the Frobenius morphism is affine and $\mathsf{F}_{m}^{*}\cO_{\PP^{n}_{k}}(j)=\cO_{\PP^{n}_{k}}(mj)$, the projection formula implies that if $\cE$ is a line bundle, then for $0<i<n$
\[
H^{i}(\PP^{n},(\Frobm\cE)(j))=H^{i}(\PP^{n},\Frobm\cE(mj))=H^{i}(\PP^{n},\cE(mj))=0
\]
and therefore $\Frobm\cE$ is a direct sum of line bundles. We plan to use a similar approach. First we find a cohomological criterion for a vector bundle $\cE$ to decompose as a direct sum 
\begin{equation}\label{Direct2}
\cE=\bigoplus_{\substack{1\leq s\leq n-1\\r\in\ZZ}}\Omega_{\PP^{n}_{k}}^{s}(-r)^{\oplus a_{r,s}}\oplus\bigoplus_{k\in\ZZ}\cO_{\PP^{n}_{k}}(k)^{\oplus b_{k}}
\end{equation}  
and then we show that this criterion is always satisfied by $\Frobm\cE$ for $m\geq m(\cE).$ Here is the precise statement. For a vector bundle $\cE$ we write $h^{i}(\PP^{n}_{k},\cE)=\dim_{k}H^{i}(\PP^{n}_{k},\cE)$. We set
\[
\cH(\cE)=\{(r,s)\in\ZZ\times\{1,\dots,n-1\}:h^{s}(\PP^{n},\cE(r))\neq0\}.
\]
Consider the following property of a vector bundle $\cE:$
\begin{equation}\tag{$\dagger$}
\textit{if $(r,s)\in\cH(\cE)$ and $t>0$ then $(r+t,s-t+1)\notin\cH(\cE)$}
\end{equation}
Our cohomological criterion is the following.

\begin{thm}\label{Thm2}
Let $k$ be an algebraically closed field of arbitrary characteristic and let $\cE$ be a vector bundle on $\PP^{n}_{k}$ that satisfies condition $(\dagger)$. Then $\cE$ decomposes into a direct sum as in the equation (\ref{Direct2}), with $a_{r,s}=h^{s}(\PP^{n}_{k},\cE(r)).$
\end{thm}

The proof of Theorem \ref{Thm2} follows from a careful analysis of Beilinson's spectral sequence. Although condition $(\dagger)$ does not seem very natural it should be seen as a generalization of Horrocks' criterion. Indeed, if $\cH(\cE)=\emptyset$ then condition $(\dagger)$ is trivially satisfied and in the decomposition (\ref{Direct2}) we must have $a_{i,j}=0$ for all $i,j$. Therefore $\cE$ splits as a direct sum of line bundles.
 
\subsection*{Acknowledgements}
This work was supported by the project KAPIBARA funded by the European Research Council (ERC) under the European Union's Horizon 2020 research and innovation programme (grant agreement No 802787). The author would like to thank P.\ Achinger, A.\ Langer, C.\ Raicu, K. VandeBogert and J.\ A.\ Wiśniewski for the interesting discussions and for their helpful comments. In particular we thank P.\ Achinger for his help with Lemma \ref{ToricLemma}.

\section{Preliminaries}

In this section we fix the notation and recall some classical theorems concerning vector bundles on projective spaces.\newline

From now on we fix an algebraically closed field $k$. We will write $\PP^{n}=\PP^{n}_{k}$ and we will write $\mathsf{F}_{m}:\PP^{n}\to\PP^{n}$ for the morphism that raises projective coordinates to the $m$-th power. We will also denote $\Omega^{i}=\Omega^{i}_{\PP^{n}_{k}}=\bigwedge^{i}\Omega_{\PP^{n}_{k}}.$ By convention $\Omega^{r}(r)$ is understood to be zero if $r\notin[0,n].$

\begin{thm}[{Bott's Formula, \cite[p.4]{Okonek}}]\label{Bott} 
Let $1\leq q \leq n-1$. Then

\begin{equation}
h^{q}(\PP^{n},\Omega^{i}(j))=
\begin{cases}
1 &\textnormal{ if } $i=q$ \textnormal{ and } $j=0$,\\
0 &\textnormal{otherwise}.
\end{cases}
\end{equation}

\end{thm}

\begin{lem}[{Beilinson, \cite[Lemma 2]{Beilinson}}]\label{BeiLemma}
We have $\textnormal{Ext}^{i}(\Omega^{r}(r),\Omega^{s}(s))=0$ for all $i>0$ and all  $r,s$.
\end{lem}

\begin{thm}[{Beilinson's spectral sequence, \cite[Proposition 8.28]{Huybrechts}}]\label{Bei}
Let $\cE$ be a vector bundle on $\PP^{n}.$ Then there is a spectral sequence of vector bundles $E^{r,s}_{t}$ on $\PP^{n}$ such that

\begin{enumerate} 

\item[i)] $E^{r,s}_{1}=H^{s}(\PP^{n},\cE(r))\otimes\Omega^{-r}(-r)$,

\item[ii)] $E_{\infty}^{r,s}=0$ for $r\neq-s$,

\item[iii)] $\cE$ has a filtration $0=\cF^{1}\subset \cF^{0}\subset\dots\subset \cF^{-n}=\cE$ where $\cF^{i}/\cF^{i+1}=E_{\infty}^{i,-i}$.

\end{enumerate}
\end{thm}

The following theorem of Horrocks, which has already appeared in the introduction, is very well known. Since we will use it several times throughout this paper we recall the precise statement.

\begin{thm}[{Horrocks' criterion, \cite[Theorem 2.3.1]{Okonek}}]\label{Horrocks1}
Let $\cE$ be a vector bundle on $X$. Then $\cE$ is a direct sum of line bundles if and only if for all integers $k$ and all $1\leq j\leq n-1$ we have $h^{j}(\PP^{n},\cE(k))=0$.
\end{thm}

At some point we will also need the following easy observation.

\begin{lem}\label{exty}

Consider an exact sequence of vector bundles
\[
0\to \cF\to\cF_{0}\to\dots\to\cF_{n}\to0.
\]
If for all $i>0$ and all $t=0,\dots,n$ we have $\textnormal{Ext}^{i}(\cF_{t},\cG)=0$ then for all $i>0$ we have $\textnormal{Ext}^{i}(\cF,\cG)=0.$
\end{lem}

\begin{proof}
By induction on $n.$
\end{proof}

\section{Proof of Theorem \ref{Thm2}}

In this section we prove Theorem \ref{Thm2}. For clarity we split the proof into several lemmas. We start by investigating how property $(\dagger)$ affects the shape of Beilinson's spectral sequence.

\begin{lem}\label{DaggerLemma1}
Let $\cE$ be a vector bundle on $\PP^{n}$ with property $(\dagger).$ Then for every $1\leq s\leq n-1$, the following hold.

\begin{enumerate}
\item[a)]  If $(r,s)\notin\cH(\cE)$ then $E^{r,s}_{1}=0$ and therefore $E^{r,s}_{t}=0$ for all $t\geq1.$

\item[b)] The only possibly nonzero morphism from $E^{r,s}_{t}$ (where $t$ runs over all positive integers) is $E^{r,s}_{s+1}\to E^{r+s+1,0}_{s+1}$. If $r+s\geq0$ then there are no nonzero morphisms from $E^{r,s}_{t}$.

\item[c)] The only possibly nonzero morphism to $E^{r,s}_{t}$ (where $t$ runs over all positive integers) is $E^{r+s-n-1,n}_{n-s+1}\to E^{r,s}_{n-s+1}$. If $r+s\leq0$ then there are no nonzero morphims to $E^{r,s}_{t}$.

\item[d)] We have $E^{-s,s}_{\infty}=E^{-s,s}_{1}$.

\item[e)] For any $k\geq0$ and any $t\geq2$ we have an exact sequence
\begin{equation}\label{ses1}
0\to E_{1}^{-k-t,t-1}\to E_{t}^{-k,0}\to E_{t+1}^{-k,0}\to0
\end{equation}
\end{enumerate}
\end{lem}
\begin{figure}[h]\label{SpectralFigure}
    \centering
    \includegraphics{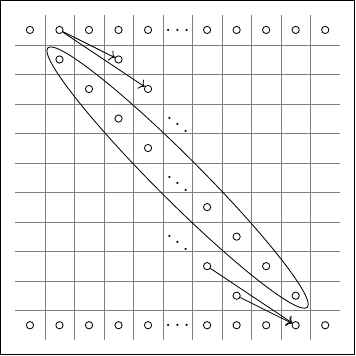}
    \caption{This graphic illustrates why Lemma \ref{DaggerLemma1} holds. Beilinson's spectral sequence is concentrated in the square $\{(r,s):-r,s\in\{0,1,\dots,n\}\}$. If $\cE$ has property ($\dagger$) then there are no nontrivial arrows between elements with $1\leq s\leq n-1$. In particular any entry above (resp. below) the diagonal admits only nonzero arrows from (resp. to) the elements in the top (resp. bottom) row. Circled elements on the diagonal have to satisfy $E^{-r,r}_{1}=E^{-r,r}_{\infty}$.}
    \label{fig:enter-label}
\end{figure}
\begin{proof}[Proof of Lemma \ref{DaggerLemma1}]
Since $E^{r,s}_{1}=H^{s}(\PP^{n},\cE(r))\otimes\Omega^{-r}(-r)$ we see that $E^{r,s}_{1}=0$ whenever $H^{s}(\PP^{n},\cE(r))=0$. Since $1\leq s\leq n-1$ we have $H^{s}(\PP^{n},\cE(r))=0$ if and only if $(r,s)\notin\cH(\cE).$ Thus $a)$ follows.
For the part $b)$ we may assume that $E^{r,s}_{1}$ is nonzero, in particular $(r,s)\in\cH(\cE).$ Then the condition $(\dagger)$ implies that $E^{r+t,s-t+1}_{t}=0$ as long as $s-t+1\neq0.$ Moreover if $r+s\geq0$ then $E^{r+s+1,0}_{1}=0$.
Proof of $c)$ is analogous to that of $b)$.
Part $d)$ follows from $b)$ and $c)$ because there are no nonzero morphisms from or to $E^{-s,s}_{t}$.
For the proof of $e)$ notice that The sequence
\[
E_{t}^{-k-t,t-1}\to E_{t}^{-k,0}\to E_{t+1}^{-k,0}\to0
\]
is always exact, so we only need to show that $E_{t}^{-k-t,t-1}=E_{1}^{-k-t,t-1}$ and that the left arrow is injective. We may assume that $t\leq n$ since otherwise $-k-t<-n$ and $E_{1}^{-k-t,t-1}=0$. Then since $-k-t+t-1=-k-1\leq0$ there is no nonzero map to $E_{m}^{-k-t,t-1}$ (where $m$ runs over all positive integers) by $c)$ and thus by $b)$ we have $E_{1}^{-k-t,t-1}=E_{t}^{-k-t,t-1}$. Injectivity is proven in a similar way. By $b)$ and $c)$ there is no nonzero morphisms to nor from $E_{m}^{-k-t,t-1}$ for $m\geq t+1$ and thus
\[
\ker(E_{t}^{-k-t,t-1}\to E_{t}^{-k,0})=E^{-k-t,t-1}_{\infty}=0,
\]
 where the second equality holds because if $k+t=t-1$ then $t=-1$ contrary to the initial assumption.
\end{proof}

The following lemma is crucial in the proof of Theorem \ref{Thm2}.

\begin{lem}\label{E00}
Let $\cE$ be a vector bundle which satisfies $(\dagger)$ and let $E^{0,0}_{\infty}$ be defined by Beilinson's sequence for $\cE.$ Then $\textnormal{Ext}^{i}(\Omega^{k}(k),E^{0,0}_{\infty})=0$ for all $k\geq0$ and $i>0.$
\end{lem}

\begin{proof}
The proof takes some computations. For clarity we split it into several steps. For the purpose of this proof we will say that a vector bundle $\cF$ has property $\cP$ if $\textnormal{Ext}^{i}(\Omega^{k}(k),\cF)=0$ for all $k\geq0$ and $i>0.$ It follows from the construction of Beilinson's spectral sequence and from Lemma \ref{BeiLemma} that any element on the first page has $\cP$ and we will use this fact several times in the proof. \newline

\textit{Step 1: $E_{2}^{-k,0}$ has $\cP$ for $k\geq1$}. Since $E_{\infty}^{-k,0}=0$ we conclude that $E_{t}^{-k,0}=0$ for some $t\geq2$ and in particular it has $\cP.$ We now prove by reversed induction that if $E_{t}^{-k,0}$ has $\cP$ then so does $E_{t-1}^{-k,0}.$  It is clear that an extension of bundles with $\cP$ also has $\cP$. Therefore the claim follows from the part $e)$ of Lemma \ref{DaggerLemma1}.\newline

\textit{Step 2:  $E_{2}^{0,0}$ has $\cP$.}  Let $\varphi_{k}:E^{-k,0}_{1}\to E^{-k+1,0}_{1}$ denote maps on the first page of Beilinson's spectral sequence and denote $A_{k}=\ker \varphi_{k}$ and $B_{k}=\textnormal{im }\varphi_{k}.$ We have exact sequences 
\begin{equation}\label{S1}
0\to B_{k+1}\to A_{k}\to E_{2}^{-k,0}\to 0
\end{equation}
and
\begin{equation}\label{S2}
0\to A_{k}\to E_{1}^{-k,0}\to B_{k}\to 0
\end{equation}
 We prove by reverse induction on $k$ that both $A_{k}$ and $B_{k}$ have $\cP$ for $k\geq1.$ This is clear for $k>n.$ If $B_{k+1}$ has $\cP$ then so does $A_{k}$ by the exact sequence (\ref{S1}) and the first step. If in an exact sequence $0\to\cF'\to\cF\to\cF''\to0$ both $\cF$ and $\cF'$ have $\cP$ then so does $\cF''.$ Thus it follows from the exact sequence (\ref{S2}) that if $A_{k}$ has $\cP$ then so does $B_{k}.$ Now for $k=0$ we have a short exact sequence
\[
0\to B_{1}\to E^{0,0}_{1}\to E_{2}^{0,0}\to 0
\]
and thus $E^{0,0}_{2}$ has $\cP.$\newline

\textit{Step 3: $E^{0,0}_{\infty}$ has $\cP$.} For that it suffices to show that $E^{0,0}_{t}$ has $\cP$ for all $t\geq2.$ We prove this by induction on $t$. For $t=2$ this is already done. Then we use short exact sequences (\ref{ses1}) with $k=0$
\[
0\to E^{-t,t-1}_{1}\to E^{0,0}_{t}\to E_{t+1}^{0,0}\to0.
\]
Now $E^{0,0}_{t}$ has $\cP$ by the inductive assumption and $E^{-t,t-1}_{1}$ has $\cP$ because every element on the first page has $\cP$. We conclude that $E_{t+1}^{0,0}$ has $\cP$ and we are done.
\end{proof}

\begin{lem}\label{SplitFiltration}
Let $\cE$ be a vector bundle on $\PP^{n}$ with property $(\dagger)$. Assume moreover that $\cH(\cE)\subset\{r+s\leq0\}$. Then the filtration on $\cE$ induced by Beilinson's spectral sequence splits non-canonically, yielding a direct sum decomposition.
\begin{equation}\label{Direct3}
\cE=\bigoplus_{j=0}^{n}E^{-j,j}_{\infty}=E^{0,0}_{\infty}\oplus E^{-n,n}_{\infty}\oplus\bigoplus_{d=1}^{n-1}\Omega^{d}(d)^{\oplus a_{d}}
\end{equation}
where $a_{d}=h^{d}(\PP^{n},\cE(-d))$.
\end{lem}

\begin{proof}
The second equality follows from Lemma \ref{DaggerLemma1} d) and the construction of Beilinson's spectral sequence. To prove the first equality we have to show that 
\begin{equation}\label{extvanish}
\begin{split}
\textnormal{Ext}^{1}(E^{-i,i}_{\infty},E^{-j,j}_{\infty})=0\textnormal{ }& \textnormal{ if }i>j
\end{split}
\end{equation}
Then from short exact sequences
\[
0\to\cF^{-i+1}\to\cF^{-i}\to E^{-i,i}_{\infty}\to0
\]
we can conclude by induction on $i$ that $\cF^{-i}=\bigoplus_{j=0}^{i}E^{-j,j}_{\infty}$ and in particular $\cE=\cF^{-n}=\bigoplus_{j=0}^{n}E^{-j,j}.$\newline

We now prove (\ref{extvanish}). For $1\leq i\leq n-1$ we have $E^{-i,i}_{\infty}=E^{-i,i}_{1}=\Omega^{i}(i)^{\oplus k_{i}}$ by Lemma \ref{DaggerLemma1} d). Then Lemma \ref{BeiLemma} together with Lemma \ref{E00} imply that the equality (\ref{extvanish}) holds if $0\leq j<i\leq n-1.$ To deal with $i=n$ note that because $\cH(\cE)\subset\{r+s\leq0\}$ there is no nonzero arrows from (and to) $E^{*,n}_{t}$ for $t\geq2$. In particular we have $E^{-n,n}_{\infty}=E^{-n,n}_{2}$ and $E^{-n+r,n}_{2}=0$ for $r>0$. Therefore $E_{\infty}^{-n,n}$ is fitted into a long exact sequence
\begin{equation}\label{resolution}
0\to E^{-n,n}_{\infty}\to E^{-n,n}_{1}\to E^{-n-1,n}_{1}\dots\to E^{0,n}_{1}\to0
\end{equation}
Since $E_{1}^{-k,n}=\Omega^{k}(k)^{\oplus},$ by Lemma \ref{BeiLemma} and Lemma \ref{E00} we have $\textnormal{Ext}^{i}(E_{1}^{-k,n},E^{-j,j}_{\infty})=0$ for $i>0$. Therefore we can use Lemma \ref{exty} and resolution (\ref{resolution}) to conclude that equality (\ref{extvanish}) holds for $i=n$.
\end{proof}
We now prove  Theorem \ref{Thm2}\newline
\begin{proof}[Proof of Theorem \ref{Thm2}]
Let us denote $\rho(\cE)=\#\cH(\cE)$. Note that this number is finite, because since $\cE$ is a vector bundle by Serre's vanishing theorem and Serre's duality for $|k|\gg0$ we have $h^{i}(\PP^{n},\cE(k))=0$ for $i\neq 0,n.$ The proof is by induction on $\rho(\cE)$.\newline

If $\rho(\cE)=0$ then by Horrocks' criterion (\ref{Horrocks1}) $\cE$ splits as a direct sum of line bundles and the decomposition (\ref{Direct2}) holds with $a_{r,s}=0$.\newline

Now we assume that $\rho(\cE)>0$. It is clear that $(r,s)\in\cH(\cE)$ if and only if $(r-k,s)\in\cH(\cE(k)).$ Therefore $\cE$ has property $(\dagger)$ if and only if $\cE(k)$ has property $(\dagger)$. It is also clear that if $\cE(k)$ has decomposes as in (\ref{Direct2}) then the twist of this decomposition by $-k$ is the desired decomposition for $\cE$. Since $\cH(\cE)$ is finite and nonempty, after twisting $\cE$ by some $k$ we may assume that $\cH(\cE)\subset\{r+s\leq0\}$ and that there exists $(-d,d)\in\cH(\cE)\cap\{r+s=0\}$. Then from Lemma \ref{SplitFiltration} we obtain a direct sum decomposition
\begin{equation}\label{E'}
\cE=\Omega^{d}(d)^{\oplus h^{d}(\PP^{n},\cE(-d))}\oplus\cE'
\end{equation}
with $\cE'=\bigoplus_{j\neq d}E^{-j,j}_{\infty}.$ Since $\cE'$ is a direct summand of $\cE$ we have $\cH(\cE')\subset\cH(\cE)$ and in particular $\cE'$ has property $(\dagger).$ Moreover $\rho(\cE')<\rho(\cE)$. Indeed, from Bott formula (\ref{Bott}) and (\ref{E'}) we have
\[
h^{d}(\PP^{n},\cE'(-d))=h^{d}(\PP^{n},\cE(-d))-h^{d}(\PP^{n},\cE(-d)).h^{d}(\PP^{n},\Omega^{d})=0
\]
and thus $(-d,d)\in \cH(\cE)\setminus\cH(\cE')$. Therefore we may apply inductive assumption to $\cE'$ and we are done.
\end{proof}

Theorem \ref{Thm2} allows to recover some classical splitting criteria. For example the following theorem of Horrocks (\cite[Expose 11]{Horrocks}) is a straightforward corollary of Theorem \ref{Thm2}. In fact deriving Theorem \ref{Horrocks2} from Beilinson's spectral sequence seems to be considered \textit{folklore} and is well known to the experts.

\begin{thm}[Horrocks]\label{Horrocks2}
Let $\cE$ be a vector bundle on $\PP^{n}$ and assume that for some $1\leq d \leq n-1$ we have
\begin{enumerate}

\item[a)] $h^{d}(\PP^{n},\cE)=r$.

\item[b)]  If $k\in\ZZ$ and $1\leq j\leq n-1$ then $h^{j}(\PP^{n},\cE(k))=0$, unless $j=d$ and $k=0$.

\end{enumerate}
Then there exist line bundles $\cL_{1},\dots,\cL_{s}$ such that
\[
\cE=(\Omega^{d})^{\oplus r}\oplus\cL_{1}\oplus\dots\oplus\cL_{s}
\]
\end{thm}

\begin{proof}
Under these assumptions we have $\#\cH(\cE)=1$ so $\cE$ has property $(\dagger)$ and we can apply Theorem \ref{Thm2}.
\end{proof}

\section{Proof of Theorem \ref{MainThm}}
In this section we prove Theorem \ref{MainThm}. The proof is a rather straightforward application of Theorem \ref{Thm2} together with the projection formula applied to the Frobenius morphism. Recall that since the $\Fm$ is affine we have
\begin{equation}\label{FrobeniusCohomology}
H^{j}(\PP^{n},(\Frobm\cE)(k))=H^{j}(\PP^{n},\Frobm(\cE(mk)))=H^{j}(\PP^{m},\cE(mk))
\end{equation}
We also recall that we denoted by $m(\cE)$ a positive integer such that for all $m\geq m(\cE)$ the twists $\cE(m)$ have only global section and the twists $\cE(-m)$ have only cohomology in degree $n$.

\begin{proof}[Proof of Theorem \ref{MainThm}]
First of all we note that if $m\geq m(\cE)$ then $(\Frobm\cE)(k)$ has cohomology only in degrees $0$ and $n$. This is clear from (\ref{FrobeniusCohomology}). This implies property $(\dagger)$ because $\cH(\Frobm\cE)$ is concentrated in a vertical line $\{(0,s):s\in\{1,\dots,n-1\}\}$. By Theorem \ref{Thm2} we obtain a decomposition
\begin{equation}\label{D1}
\Frobm\cE=\bigoplus_{\substack{1\leq s\leq n-1\\r\in\ZZ}}\Omega^{s}(-r)^{\oplus a_{r,s}}\oplus\bigoplus_{k\in\ZZ}\cO(k)^{\oplus b_{k}}.
\end{equation}
Moreover we must have $a_{r,s}=0$ for $r\neq0$ since otherwise $(\Frobm\cE)(r)$ would have nonzero cohomology in degree $s$. Thus decomopostion (\ref{D1}) may be rewritten as
\begin{equation}\label{D2}
\Frobm\cE=\bigoplus_{i=1}^{n-1}(\Omega^{i})^{\oplus a_{i}}\oplus\bigoplus_{k\in\ZZ}\cO(k)^{\oplus b_{k}},
\end{equation}
with $a_{i}=h^{i}(\PP^{n},\Frobm\cE)=h^{i}(\PP^{n},\cE).$ To obtain decomposition $(\ref{Direct1})$ we only need to show that $b_{k}=0$ for $k\notin\{0,-1,\dots,-n,-n-1\}$. We can then rewrite $\cO_{\PP^{n}}=\Omega^{0}$ and $\cO_{\PP^{n}}(-n-1)=\Omega^{n}$ and use classical formulas for cohomology of line bundles on $\PP^{n}$ to check that %
\[
b_{0}=h^{0}(\PP^{n},\Frobm\cE)=h^{0}(\PP^{n},\cE)
\]
and similarly $b_{-n-1}=h^{n}(\PP^{n},\cE)$. So assume that we have $b_{k}\neq 0$ for some $k>0$. Then $(\Frobm\cE)(-1)$ would have a nonzero global section and this together with (\ref{FrobeniusCohomology}) contradicts our choice of $m$. Similarly $b_{k}=0$ for $k<-n-1$ because $(\Frobm\cE)(1)$ has no cohomology in degree $n$.
\end{proof}

Given a vector bundle $\cE$ it is natural to ask what is the minimal positive integer $b=b(\cE)$ such that $\mathsf{F}_{b*}\cE$ decomposes as in (\ref{Direct1}) or (\ref{Direct2}). In general it seems to be difficult to express $b(\cE)$ in terms of known invariants of vector bundles. Below we discuss some examples.

\begin{ex}
Let $\cE$ be any vector bundle. Then by the projection formula $\cE$ is a direct summand of $\Frobm\mathsf{F}_{m}^{*}\cE$ and therefore if $\cE$ was not already of form (\ref{Direct1}) then $b(\mathsf{F}_{m}^{*}\cE)>m$. In particular $b(\cE)$ may be arbitrary large for fixed $\rk\cE\gg1$ and $n>1$.    
\end{ex}

\begin{ex}
In general $m(\cE)$ is not minimal among these $m$ for which $\Frobm(\cE)$ decomposes as in (\ref{Direct1}). let $n\geq4$ and $\cE=\Omega^{3}(3).$ Then $m(\cE)\geq4$ because $\cE(-3)$ has nontrivial $h^{3}$ by the Bott formula (\ref{Bott}). On the other hand it follows from the same formula and Horrocks' criterion (\ref{Horrocks1}) that $\mathsf{F}_{2*}\cE$ splits as a direct sum of line bundles. For example on $\PP^{5}$ we have
\[
\mathsf{F}_{2*}(\Omega^{3}(3))=\cO_{\PP^{5}}(-1)^{\oplus 84}\oplus\cO_{\PP^{5}}(-2)^{\oplus 216}\oplus\cO_{\PP^{5}}(-3)^{\oplus 20}
\]
\end{ex}

\section{Theorem of Klyachko}

We now present an application of our Theorem \ref{MainThm} to the theory of toric vector bundles. Let $\TT$ be the $n$-dimensional torus acting on $\PP^{n}$ in the natural way. After fixing coordinates in $\PP^{n}$ and $\TT$ we may assume that the action of $\TT$ is given by
\begin{equation}\label{T-action}
(t_{1},\dots,t_{n}).[x_{0}:x_{1}:\dots:x_{n}]=[x_{0}:t_{1}x_{1}:\dots:t_{n}x_{n}],
\end{equation}
The open orbit of this action is
\[
\cU=\TT.[1:1:\dots:1]=\textnormal{Spec }k[u_{1}^{\pm1},\dots,u_{n}^{\pm1}]
\]
A \textit{toric vector bundle} on $\PP^{n}$ is a vector bundle $\cE$ together with a fixed linearization for the action (\ref{T-action}) in the sense of \cite[$\S$3. Definition 1.6]{Mumford} (the definition given there for line bundles generalizes trivially to vector bundles of arbitrary rank).  Toric bundles on smooth complete toric varieties have been studied intensively by A.\ A.\ Klyachko (see for example \cite{Klyachko} for an exposition of his results). His main result states that the category of toric vector bundles (on a smooth complete toric variety) is equivalent to the category of vector spaces endowed with a family of filtrations and satisfying some extra conditions. Using this combinatorial description he was able to prove numerous interesting properties of toric vector bundles, for example the following vanishing theorem for cohomology of toric vector bundles on projective spaces.

\begin{thm}[{Klyachko, \cite[ Corollary 4.2]{Klyachko}}]\label{KlyachkoThm}
Let $\cE$ be a toric vector bundle on $\PP^{n}$. If $\rk\cE<{n \choose r}$ then $H^{r}(\PP^{n},\cE)=0$.
\end{thm}

Since Klyachko's formalism is very combinatorial it may be interesting to find a more direct proof of Theorem \ref{KlyachkoThm}. We note that a short proof of this theorem in case $\rk\cE=2$ (which is particularly interesting due to its connection to Hartshorne's conjecture) has recently been found by D.\ Stapleton in \cite{Stapleton}. Below we present a short noncombinatorial proof valid for arbitrary rank of $\cE$. We first prove Lemma \ref{ToricLemma} below, which states that if $\cE$ is toric then $\Frobm\cE$ admits a natural direct sum decomposition and then we compare this decomposition with decomposition (\ref{Direct1}) to conclude Theorem (\ref{KlyachkoThm}).

\begin{lem}\label{ToricLemma}
Let $\cE$ be a toric vector bundle on $\PP^{n}$ and let $m$ be a positive integer coprime to $\textnormal{char }k$. Then $\Frobm\cE=\bigoplus_{i=1}^{m^{n}}\cE_{i}$, where $\cE_{i}$ are vector bundles and $\rk\cE=\rk\cE_{i}$.
\end{lem}

\begin{proof}
For simplicity we will write $\mathsf{F}=\mathsf{F}_{m}$. Let $G=\mu_{m}^{n}\subset\TT$ be the subgroup of $\TT$ defined as
\[
G=\{(t_{1},\dots,t_{n})\in\TT:t_{i}^{m}=1\textnormal{ for all }i\}
\]
Since $m$ is coprime to $\textnormal{char }k$, we have an isomorphism
\begin{equation}\label{Characters}
\textnormal{Hom}(G,k^{\times})=(\ZZ/m\ZZ)^{\oplus n}
\end{equation}
which is given by
\[
(r_{1},\dots,r_{n})\mapsto[(t_{1},\dots,t_{n})\mapsto \prod_{i=1}^{n}t_{i}^{r_{i}}]
\]
It is clear that $\mathsf{F}$ coincides with the quotient of $\PP^{n}$ by the action of $G$. We leave it to the reader to verify that this is a \textit{geometric quotient}, i.e., that $\mathsf{F}$ is open, $G$ acts transitively on the fibers of $\mathsf{F}$, and that $(\Frob\cO_{\PP^{n}})^{G}=\cO_{\PP^{n}}$. Now, if $\cE$ is a toric vector bundle then it is in particular equivariant with respect to the $G$-action and therefore $\Frob\cE$ is naturally a sheaf of $\cO_{\PP^{n}}[G]$-modules. Since $m$ is coprime to the characteristic of $k=\overline{k}$ the action of $G$ is diagonalizable and we obtain a decomposition 
\[
\Frob\cE=\bigoplus_{\chi\in\textnormal{Hom}(G,k^{\times})}\cE_{\chi},
\]
where for any open subset $U\subset\PP^{n}$ we have
\[
\cE_{\chi}(U)=\{m\in\cE(U):g.m=\chi(g)m\textnormal{ for all }g\in G\}.
\]
We want to show that $\rk\cE=\rk\cE_{\chi}$. For that it suffices to compare ranks on some nonempty open subset of $\PP^{n}$, for example on the open orbit $\cU$. If $\cE$ is a toric bundle on $\PP^{n}$ then $\cE_{|\cU}=\cO_{\cU}^{\oplus\rk\cE}$ as toric bundles, where the toric structure on $\cO_{\cU}$ is the one induced by the left action of $\TT$ on itself, i.e., is given by
\begin{equation}\label{G-action}
(t_{1},\dots,t_{n}).u_{1}^{a_{1}}\dots u_{n}^{a_{n}}=t_{1}^{a_{1}}\dots t_{n}^{a_{n}}u_{1}^{a_{1}}\dots u_{n}^{a_{n}}
\end{equation}
Clearly we have $\mathsf{F}^{-1}\cU=\cU$ and thus
\[
(\Frob\cE)_{|\cU}=(\mathsf{F}_{|\cU})_{*}\cE_{|\cU}=(\mathsf{F}_{|\cU})_{*}\cO_{\cU}^{\oplus\rk\cE}.
\]
Therefore we only need to check that 
\begin{equation}\label{StrucureSheaf}
\Frob\cO_{\cU}=\bigoplus_{\chi\in\textnormal{Hom}(G,k^{\times})}\cO_{\chi}
\end{equation}
is a direct sum decomposition into line bundles, because then
\[
\cE_{\chi|\cU}=\cO_{\chi}^{\oplus\rk\cE}.
\]
To see that $\cO_{\chi}$ is a line bundle note that the $\cO_{\cU}$-module structure on $\Frob_{|\cU}\cO_{\cU}$ corresponds to the decomposition of $k[u_{1}^{\pm1},\dots,u_{n}^{\pm1}]$ as a $k[u_{1}^{\pm m},\dots, u_{n}^{\pm m}]$-module, which is given by
\begin{equation}\label{Laurent}
k[u_{1}^{\pm1},\dots,u_{n}^{\pm1}]=\bigoplus_{0\leq r_{i}\leq m-1}u_{1}^{r_{1}}\dots u_{n}^{r_{n}}k[u_{1}^{\pm m},\dots, u_{n}^{\pm m}]
\end{equation}
Comparing this with (\ref{Characters}) and (\ref{G-action}) we see, that if $\chi$ is a character corresponding to $(r_{1},\dots,r_{n})$ then
\[
\cO_{\chi}=u_{1}^{r_{1}}\dots u_{n}^{r_{n}}k[u_{1}^{\pm m},\dots, u_{n}^{\pm m}].
\]
Therefore decompositions (\ref{StrucureSheaf}) and (\ref{Laurent}) coincide and we are done.
\end{proof}

\begin{rmk}
If $m$ is not coprime to $\textnormal{char }k$ then Lemma \ref{ToricLemma} is still valid. If $\textnormal{char }k=p$ and $m=p^{r}$ then it is consequence of a more general theorem of R.\ B{\o}gvad \cite[Theorem 1]{Bogvad} and in this case components $\cE_{i}$ are again toric. In general $m=p^{r}m_{1}$, where $m_{1}$ is comprime to $p$ and we can factor $\mathsf{F}_{m}=\mathsf{F}_{m_{1}}\mathsf{F}_{p^{r}}$.
\end{rmk}

We now give a proof of Theorem \ref{KlyachkoThm}.

\begin{proof}[Proof of Theorem \ref{KlyachkoThm}]
Let $\cE$ be a toric vector bundle on $\PP^{n}$ of rank $\rk\cE<{n \choose r}$. We can find $m$ that is coprime to $\textnormal{char }k$ and is large enough so that decomposition (\ref{Direct1}) exists. Comparing this decomposition with the decomposition from Lemma \ref{ToricLemma} we obtain equality
\[
\Frobm\cE=\bigoplus_{i=0}^{n}(\Omega^{i})^{\oplus a_{i}}\oplus\bigoplus_{j=1}^{n}\cO(-j)^{\oplus b_{j,m}}=\bigoplus_{i=1}^{m^{n}}\cE_{i}
\]
with $\rk\cE_{i}<{n \choose r}.$ It is well known (and clear from Theorem \ref{Horrocks2}) that $\Omega^{i}$ are indecomposable vector bundles. An old theorem of M.\ Atiyah \cite[Theorem 3]{Atiyah} implies that every vector bundle on $\PP^{n}$ admits a unique decomposition into a direct sum of indecomposable vector bundles (the so called \textit{Remak decomposition}). Thus the decomposition on the left hand side must refine the decomposition on the right hand side and in particular every indecomposable summand on the left hand side has rank $<{n \choose r}$. Since $\rk\Omega^{r}={n \choose r}$ we obtain $h^{r}(\PP^{n},\cE)=a_{r}=0$ and we are done.

\end{proof}


\bibliographystyle{plain}
\bibliography{BibliographyFrobenius}

\end{document}